\documentclass{amsart}
\usepackage{amsmath}
\usepackage{graphicx} 
\usepackage{amsfonts} 
\usepackage{amssymb}
\usepackage[font=small,format=hang,labelfont={sf,bf}]{caption}
\usepackage{color}
\usepackage{mathrsfs}
\usepackage{epsfig}
\usepackage{subfig}
\usepackage{url}
\usepackage{varioref}
\theoremstyle{plain}
\newtheorem{theorem}{Theorem}[section]

\newtheorem{lemma}[theorem]{Lemma}
\newtheorem{proposition}[theorem]{Proposition}

\newtheorem{definition}[theorem]{Definition}
\theoremstyle{definition}
\theoremstyle{remark}
\numberwithin{equation}{section}

\newcommand{\f}{\varphi}

\newcommand{\e}{\varepsilon}

\newcommand{\R}{\mathbb{R}}
\newcommand{\N}{\mathbb{N}}

\newcommand{\spt}{\mathrm{supp}}

\newcommand{\supp}{\mathrm{supp}\,}
\newcommand{\M}{\mathcal{M}}

\newcommand{\weakly}{\rightharpoonup}           
\newcommand{\weakstar}{\stackrel{*}{\weakly}}   

\newcommand{\jj}{\hat J}

\newcommand{\dist}{\text{dist}}

\usepackage{latexsym}
 
\title[] {Attractive Riesz potentials acting on \\ hard spheres} 

\author[A. Kubin]
{A. Kubin}
\address[Andrea Kubin]{Dipartimento di Matematica ``Guido Castelnuovo'', Sapienza Universit\`a di Roma, P.le Aldo Moro 5, I-00185 Roma, Italy}
\email[A. Kubin]{kubin@mat.uniroma1.it}

\author[M. Ponsiglione]
{M. Ponsiglione}
\address[Marcello Ponsiglione]{Dipartimento di Matematica ``Guido Castelnuovo'', Sapienza Universit\`a di Roma, P.le Aldo Moro 5, I-00185 Roma, Italy}
\email[M. Ponsiglione]{ponsigli@mat.uniroma1.it}

\begin{document}

\begin{abstract}
In this paper we introduce a model for hard spheres interacting through attractive Riesz type potentials, and we study its thermodynamic limit.  We show that the tail energy enforces optimal packing and round macroscopic shapes.

\end{abstract}

\maketitle
\tableofcontents
\section{Introduction}
In this paper we introduce and analyze variational models for hard spheres interacting through Riesz type attractive potentials. 
The model consists in minimizing  nonlocal energies of the type
\begin{equation}\label{eneintro}
\sum_{i\neq j} K^p(|x_i-x_j|), 
\end{equation}
over all configurations of $N$ points $\{x_1,\, \ldots \, , x_N\}\subset \R^d$ satisfying $|x_i-x_j|\ge 2$ for all $i\neq j$; here
$K^p:\R^+ \to (-\infty,0]$ is a power-law attractive potential $K^p(r) \approx - \frac{1}{r^p}$ for large $r$, with $p \in (0,d+1)$. Eventually, we consider the thermodynamic limit $N\to +\infty$. 


As a consequence of the hard sphere constraint and of the attractive behaviour of the  potential, 
the ground states of the system  turn out to be optimal packed configurations of spheres filling a macroscopic set. The thermodynamic limit is described by a nonlocal energy that is a Riesz type continuous counterpart of \eqref{eneintro} for  $p\in(0,d)$; 
remarkably, in the  case $p\in[d,d+1)$
fractional perimeters  arise in the limit energy. In both cases $p\in(0,d)$  and $p\in[d,d+1)$, the optimal asymptotic shape is given (after scaling) by the Euclidean ball, and this is a consequence of the Riesz rearrangement inequality and of the fractional isoperimetric inequality, respectively. These results are obtained by providing a $\Gamma$-convergence expansion of  the energy. The method allows to consider and understand  also slight variants of the basic variational problem \eqref{eneintro}, taking into account also volume forcing terms,   possibly enforcing different optimal shapes.

The combination of the attractive potential together with the hard sphere constraint provides a basic example of long range attractive/short range repulsive  interactions. 
In this respect, the proposed model fits in the class of {\it aggregation} \cite{FHK, BCT,  CC}  and {\it crystallization} \cite{BL} problems, but with a substantial change of perspective due to the fundamental role 
played  in our model by the tail of the interaction energy.
This is the case for both integrable and  non-integrable tails, 
referred to as {\it unstable potentials} in the crystallization community \cite{BL}.
This is why in our model   crystallization  is replaced by the related but different concept of {\it optimal packing}, while 
the microscopic structure does not affect the macroscopic shape, that turns out to be the Euclidean ball. 

To explain these new phenomena, we first provide an overview of  the classical crystallization problem, focussing on two basic models in two dimensions. They are  based on minimization of an interaction energy as in \eqref{eneintro}, for some potential  $K$ that tends to infinity as $r\to 0$,  has a well at a specific length enforcing crystallization and fixing the lattice spacing (and structure), and rapidly decays to $0$ as $r\to +\infty$. 
The basic potential is provided by the Heitmann-Radin model \cite{HeRa} which consists in systems of  hard spheres 
whose pair-interaction energy is $+\infty$ if two balls overlap, it is equal to $-1$ if the balls touch each other, and $0$ otherwise. In two dimensions, and for fixed number $N$ of discs, minimizers exhibit crystalline order: the centers of the discs  lie on a subset of an  equilateral triangular lattice.  Moreover, for large $N$ the discs fit a large hexagon. The first phenomenon is referred to as {\it crystallization}, the second as {\it macroscopic Wulff-shape}. Crystallization is due to local optimization of the potential around its well: almost each particle tends to maximize the number of nearest neighbor particles. In view of the hard disc constraint, such a number is $6$.  The macroscopic Wulff shape is the result of the minimization of the number of {\it boundary} particles that have the wrong number of nearest neighbors. In this respect, the macroscopic shape minimizes an anisotropic perimeter energy; under a volume constraint, this is nothing but   the anisotropic isoperimetric problem, whose minimizer is the Wulff shape \cite{FM}. Recently, these phenomena have been analyzed in details in the solid formalism of $\Gamma$-convergence \cite{AFS, DNP}. 

A less rigid and most popular model in elasticity is given by the polynomial  Lennard-Jones type potential; the hard sphere constraint is replaced by a repulsive term which is infinite only at $0$; the only negative value in the Heitmann-Radin potential is replaced by a narrow well, while the zero-long range interaction of the Heitmann-Radin potential is replaced by a rapidly decaying tail energy. In \cite{T} it is proved that, if the well of the potential is very narrow, and the tail is a small enough lower order term, then the crystallization property is preserved in average, while the Wulff shape problem is still open. Recently, it has been proved \cite{BDP}
that  a slightly wider well in the potential favours the square lattice rather than the triangular one.  In higher dimensions the picture is much less clear. 

We pass to describe our model; since the tail energy will be predominant, it is convenient to change length-scale, introducing a parameter $\e>0$, whose inverse $\frac{1}{\e}$ represents the size of the body filled by the hard spheres. Then, in order to deal, in the thermodynamic limit, with a finite macroscopic body, we scale the spheres with $\e$.
After this scaling the potential $K^p$ becomes integrable if and only if $p\in (0,d)$. We discuss first the integrable case: we write $p= d+\sigma$ for some $\sigma\in (-d,0)$, and we introduce the corresponding potential which, up to a prefactor, becomes the function   $f_{\varepsilon}^{\sigma}:\R^+\to \overline \R$ defined  by 
\begin{equation}\label{intepo}
f_{\varepsilon}^{\sigma}(r):=
\left\{
\begin{aligned}
& +\infty & \text{ for } r \in [0,2 \varepsilon) \, , \\
& -\frac{1}{r^{d+\sigma}} & \text{ for }  r \in [2 \varepsilon,\infty) \, . 
\end{aligned}
\right.
\end{equation}
In this case the $\Gamma$-limit as $\e\to 0$ of the discrete energy \eqref{eneintro}, with $K^p=f_{\varepsilon}^{\sigma}$,  is nothing but its continuous counterpart, defined on absolutely continuous measures, whose density is bounded from above by the density of the optimal packing problem in $\R^d$ (see Theorem \ref{Gammaconvprimocaso}). 
This $\Gamma$-convergence result can be completed with suitable confining  volume forcing terms, ensuring compactness properties for minimizers. We prove that minimizers consist, in the limit as $\e\to 0$, in optimal packed configurations of balls filling a macroscopic set $E$, which is a ball whenever the volume term is radial. 

The non-integrable case is much more involved.
In this case both the tail and the core of the energy blow up as $\e\to 0$, the first being the leading term.  In order to provide a first order expansion of the energy in terms of $\Gamma$-convergence, we need to regularize the potential, neglecting the core energy. More precisely,  
we introduce a mesoscopic length-scale $r_\e \gg \e$ with $r_\e\to 0$   as $\e\to 0$ (see \eqref{condr}), and we regularize 
the potential  cutting-off  all short range interactions at   scales smaller than $r_\e$. The corresponding regularized Riesz type $p$-power-law 
potentials, with $p=d+\sigma$ and $\sigma\in[0,1)$,   
are  defined by
\begin{equation}
f_{\varepsilon}^{\sigma}(r):=
\begin{cases}
+\infty   & \text{ for } r \in [0,2 \varepsilon) \, , \\
0 &  \text{ for } r \in [2 \varepsilon,r_{\varepsilon})\, , \\
-\frac{1}{r^{d+\sigma}} &  \text{ for } r \in [r_{\varepsilon}, +\infty) \, .
\end{cases}
\end{equation} 

Then, only the tail of the interaction energy remains, and the microscopic details of the potential are neglected in the limit as $\e\to 0$.
This is consistent with the integrable case \eqref{intepo}, where the core contribution vanishes as a consequence of the only integrability of the potential.
Dividing the energy by the diverging tail contribution, we  obtain the zero order term in the  $\Gamma$-convergence expansion of the energy. This zero order $\Gamma$-limit   still enforces optimal packing on minimizing sequences, but does not determine the macroscopic limit shape. Then, we   look at the next term in the $\Gamma$-convergence expansion. This consists in removing from the total energy the infinite volume-term energy per particle, so that a finite quantity reamins, which turns out to detect the macroscopic shape. In fact, the first order $\Gamma$-limit, provided in Theorem \ref{gacofinal}, is nothing but the $\sigma$-fractional perimeter, introduced in \cite{CRS} for $\sigma\in (0,1)$, and recently  in \cite{DeLucaNovPons} also for $\sigma=0$. Such an analysis has first been provided in a continuous setting in \cite{DeLucaNovPons}; our results represent its discrete counterpart. 
Since fractional perimeters are minimized, under a volume constraint, by Euclidean balls, we deduce that, as $\e\to 0$, minimizers are given by optimal packed configurations of $\e$-spheres filling a macroscopic ball. 

The range of the parameter $\sigma\in(-d,1)$ is somehow natural, since for $\sigma < -d$ the potential becomes repulsive (and constant for $d=0$); the case $\sigma=1$ formally corresponds to the Euclidean perimeter, being the limit of $s$-fractional perimeters as $s\to 1$ \cite{ADM, CN}. The case $\sigma>1$ is unclear to us.

Our analysis represents a first attempt to provide a discrete microscopic description for a variety of continuous models characterized by nonlocal interactions. Among them we mention the Poincar\'e problem, which, for $p=1$, establishes that the ball is the optimal shape  minimizing the potential energy of a fluid \cite{Lie}. 
Moreover, the hard sphere constraint could be relaxed, providing in the limit richer models, accounting for density penalization terms,
as in the {\it rotating stars problem} \cite{Lio}, as well as several  attracting-repulsive potentials \cite{BCT}. Our analysis could also be extended to 
discrete systems describing interactions between different populations, \cite{CPT,CDNP}.

Finally, our analysis suggests the possible role of the tail energy as a new mechanism enforcing optimal packing, and hence, in some respect, crystallization. 

\subsection{Notation of the paper}
In this paper we use the following notation:  $\omega_d$ denotes the Lebesgue measure of the unit ball $B_1(0)$ of $\mathbb{R}^{d}$. 
$\mathcal{M}(\mathbb{R}^d)$ denotes the family  of Lebesgue measurable sets  $E\subset \mathbb{R}^d$, while  the corresponding Lebesgue measure will be denoted by  $|E|$. 
We set $\mathcal{M}_{f}(\mathbb{R}^d):=\{E \in \mathcal M (\mathbb{R}^d): \, |E|<+\infty\}$.  $\mathcal{M}_{b}(\mathbb{R}^{d})$ denotes the space of (non negative) finite Radon measures in $\mathbb{R}^{d}$. The Dirac delta measure centered in $x$ is denoted by $\delta_{x} $, while the Lebesgue measure by $\mathcal L^d$.  We denote with $C(\star,\cdots,\star)$ a constant that depends on $\star,\cdots,\star$; this constant can change in the steps of a proof. Finally, $\overline{\mathbb{R}}:=\mathbb{R}\cup \{-\infty,+\infty\}$.

\section{Hard spheres, optimal packing ahd empirical measures}

Here we introduce the admissible configurations of the variational model proposed in this paper, and revisit some concepts on optimal packed configurations we will need in our analysis.

\subsection{Density of optimal   packing}
\begin{definition}\label{def.T^d e C^d}
	We denote by  $\mathrm{Ad^d}$ be the class of sets  $X\subset \R^d$ such that $|x_i-x_j| \ge 2$  for all $x_i,\, x_j \in X$ with  $x_i\neq x_j$. 
	The volume density of optimal ball  packings in $\mathbb{R}^{d}$ is the constant  $C^d$  defined by 
		\begin{equation} \label{C^d0}
	C^d:= \sup_{X\in   \mathrm{Ad^d}} \limsup_{r \rightarrow + \infty} \frac{\#(X \cap r Q ) \omega_d}{\vert r Q \vert} \, , 
	\end{equation}
	where $ Q:= [0,1)^d$.
	Moreover, we say that $\mathrm{T}^{d}\subset \mathbb{R}^{d}$ is an optimal configuration for the (centers for  the unit ball) optimal packing problem
	if $\mathrm{T}^{d}\in \mathrm{Ad^d}$
	and
	\begin{equation} \label{C^d}
	\lim_{r \rightarrow + \infty} \frac{\#(\mathrm{T}^{d} \cap  r Q) \omega_d}{\vert r Q \vert} = C^d.
	\end{equation}
	\end{definition} 

In \cite{Groemer} it is proved the existence of an optimal configuration, and that in defining $C^d$ and $\mathrm{T}^{d}$, $Q$ can be replaced by 
any open bounded set $A \subset \mathbb{R}^d$ with $A\neq \emptyset$.  

Now we want to  provide  a rate of convergence in \eqref{C^d0}. To this purpose, for every $r>0$ 
let $\mathrm{Ad^d}(r Q)$ be the class of sets  $X\subset r Q$ such that $|x_i-x_j| \ge 2$  for all $x_i,\, x_j \in X$ with  $x_i\neq x_j$, and set
\begin{equation}\label{Crd}
C_r^d:= \sup_{X\in   \mathrm{Ad^d}(rQ)}  \frac{\#(X) \omega_d}{r^d} \, .
\end{equation}
It is easy to see that for all $r>0$ there exists a maximizer, denoted by ${T}^{d}_r$. 
\begin{lemma}\label{sviluppoasintoticodensità}
	There exists $C(d)>0$ such that $C^d \le C^d_{r} \le C^d +  \frac{C(d)}{r}$  for all  $r>0$.
\end{lemma}
\begin{proof}
For every $r>0$ we have  $2rQ=\cup_{i=1}^{2^d}rQ_i$ where $Q_i=Q+v_i$, $v_i \in \{0,1\}^d$. 
Let $T_r^d$ be any maximizer of \eqref{Crd}, and set 
$$
\hat T_r^d:= \{x\in T_r^d : \, \text{dist } (x, r \partial  Q) \ge 1\},
\qquad
\tilde T_{2r}^d:= \cup_{i=1}^{2^d} \hat T_r^d + v_i, \, v_i \in \{0,1\}^d. 
$$
It is easy to see that there exists a constant $c(d)$ such that $\#  T_r^d - \# \hat T_r^d \le c(d) r^{d-1}$. 
Moreover, 
$$
\max\{\# T_{2r}^d \cap r Q_i, \, i=1,\cdots,2^d\} \ge \frac{ \# T_{2r}^d}{2^d} \, .
$$
Then we have
$$
r^d C^d_{2r} = \frac{ \# T_{2r}^d}{2^d}  \le r^d C^d_{r} = \#  T_r^d \le     \frac{\# \tilde T_{2r}^d}{2^d} + c(d)  r^{d-1} \le r^d C^d_{2r} + c(d)  r^{d-1}.
$$
Therefore, for every $r>0$, $n\in\N$ we have
$$
C^d_{2^n r} \le C^d_{2^{n-1} r} \le C^d_{2^nr} + \frac{c(d)}{2^{n-1}r}, 
$$
which by iteration over $n$ yields
$$
C^d_{2^n r} \le C^d_{r} , \qquad  C^d_r \le C^d_{2^nr} + \sum_{k=1}^{n} \frac{c(d)}{2^{k-1}r}.
$$
Sending $n\to +\infty$ we deduce the claim.  
\end{proof}

\subsection{The empirical measures} 
We introduce the family of empirical measures 
\begin{equation*}
\mathcal{EM}:= \biggl\{  \sum_{i=1}^{N} \delta_{x_i}: x_{i} \neq x_{j} \; \text{for} \; i \neq j, N\in \N \biggr\} \subset \mathcal{M}_{b}(\mathbb{R}^{d}).
\end{equation*}
We consider  the space $\mathcal{M}_{b}(\mathbb{R}^d)$ endowed with the tight topology.
\begin{definition}[Tight convergence]
We say that a sequence $\{\mu_{\varepsilon}\}_{\varepsilon \in (0,1)}  \subset \mathcal{M}_{b}(\mathbb{R}^d)$  tightly converges to $\mu \in \mathcal{M}_{b}(\mathbb{R}^d)$ if $\mu_{\varepsilon} \weakstar \mu $ and $ \mu_{\varepsilon}(\mathbb{R}^d) \rightarrow \mu(\mathbb{R}^d)$, as $\varepsilon \rightarrow 0^+$.
\end{definition}
\begin{definition}
	Let $\varepsilon>0$, we define the set $\mathcal{EM_\e}\subset \mathcal{EM}$ as 
	$$
	\mathcal{EM_\e} := \biggl\{\mu \in \mathcal{EM}: \,  \mu= \sum_{i=1}^{N} \delta_{x_i} \text{ with } \vert x_i-x_j \vert \geq 2 \varepsilon \text{ for all } i \neq j\biggr\}.
$$
\end{definition}

\begin{lemma}\label{dime}
	Let $\{ \mu_{\varepsilon} \}_{\varepsilon \in (0,1)} \subset \mathcal{EM}$ with $\mu_\e\in \mathcal{EM}_\e $ for all $\e\in(0,1)$ be such that  $\frac{\varepsilon^d \omega_d}{C^d} \mu_{\varepsilon} \weakstar \mu $ for some 
	$\mu \in \mathcal{M}_{b}(\mathbb{R}^d)$, as $ \varepsilon \rightarrow 0^+$ (where $C^d$ is defined in \eqref{C^d0}).  Then, there  exists $\rho \in \mathrm{L}^{\infty}(\mathbb{R}^d, [0,1])$ such that $\mu=\rho \mathcal{L}^d$.
\end{lemma}
\begin{proof} It  is sufficient to prove that $\mu(A) \leq \vert A \vert $ for all open set $A$. By the lower semi-continuity of the total variation with respect to weak-star convergence, we have 
\begin{multline*}
\mu(A) \leq \liminf_{\varepsilon \rightarrow 0^+} \frac{\varepsilon^d \omega_d}{C^d} \mu_\varepsilon(A) 
= \liminf_{\varepsilon \rightarrow 0^+} \frac{\vert A \vert}{C^d} \frac{\omega_d \# \{  A \cap \supp(\mu_{\varepsilon})\}}{\vert \frac{A}{\varepsilon} \vert} \\
\leq \vert A \vert \lim_{\varepsilon \rightarrow 0^+} \frac{1}{C^d} \frac{\# (\mathrm{T}^d \cap \frac{A}{\varepsilon})}{\vert \frac{A}{\varepsilon} \vert}= \vert A \vert,
\end{multline*}
where the last inequality follows by \eqref{C^d0} and \eqref{C^d} with $Q$ replaced by $A$. 
\end{proof}

\begin{lemma} \label{LemmacostruzF}
For every $\rho \in \mathrm{L}^{1}(\mathbb{R}^d, [0,1]) $ there exists a sequence $\{ \mu_{\varepsilon} \}_{\varepsilon \in (0,1)} \subset \mathcal{EM}$ with $\mu_\e\in \mathcal{EM}_\e $ for all $\e\in(0,1)$ such that  $\frac{\varepsilon^d \omega_d}{C^d} \mu_\varepsilon \rightarrow \rho \mathcal{L}^d$ tightly in $\mathcal{M}_{b}(\mathbb{R}^d)$. 	
\end{lemma}

\begin{proof}
By a standard density argument, it is enough to prove the claim for $\rho= a \chi_A$ for some $a \in (0,1)$ and some open set $A\subset \R^d$.
Let $\mu_{\varepsilon}:= \sum_{i \in I_{\varepsilon}} \delta_{x_i}$ where $I_{\varepsilon}:= {\varepsilon}{a^\frac{-1}{d}} \mathrm{T}^d \cap A$.
Then, it is easy to check that $\frac{\varepsilon^d \omega_d}{C^d} \mu_\varepsilon \rightarrow a \chi_A \mathcal{L}^d$ tightly in $\mathcal{M}_{b}(\mathbb{R}^d)$. 
\end{proof}

For all $\mu:= \sum_{i=1}^{N}\delta_{x_i}$  in $\mathcal{EM}_{\varepsilon}$ we set 
\begin{equation}\label{defhm}
\hat{\mu}:=\frac{1}{C^d}\sum_{i=1}^{N}\chi_{B_\varepsilon(x_i)}.
\end{equation}

\begin{lemma}\label{lediti}
Let $\{\mu_{\varepsilon}\}_{\varepsilon \in (0,1)} \subset \mathcal{EM}$ with $\mu_{\varepsilon} \in \mathcal{EM}_{\varepsilon}$ for all $\e\in (0,1)$,  and let 
$\rho \in \mathrm{L}^1(\mathbb{R}^d,[0,1])$ be such that $\frac{\varepsilon^d \omega_d}{C^d}\mu_{\varepsilon} \rightarrow \rho \mathcal{L}^d$ tightly in $\mathcal{M}_{b}(\mathbb{R}^d)$. Then,
$\hat{\mu}_{\varepsilon} \rightarrow \rho \mathcal{L}^d$ tightly.
\end{lemma}
\begin{proof}
We observe that
\begin{equation}\label{bouhm}
\lim_{\varepsilon \rightarrow 0^+}\hat{\mu}_{\varepsilon}(\mathbb{R}^d)= \lim_{\varepsilon\rightarrow 0^+}\frac{\varepsilon^d \omega_d}{C^d} \mu_{\varepsilon}(\mathbb{R}^d)= \int_{\mathbb{R}^d}\rho(x) dx.
\end{equation} 
Therefore, up to a subsequence $\hat{\mu}_{\varepsilon} \weakstar g$ for some $g\in \mathcal{M}_{b}(\mathbb{R}^d)$. We have to prove that $g=\rho \mathcal{L}^d$. 
To this purpose,
notice that for all $\varphi \in \mathcal{C}^1_{c}(\mathbb{R}^d)$ we have
\begin{equation*}
\begin{split}
\vert \hat{\mu}_{\varepsilon}(\varphi)-\rho\mathcal{L}^d(\varphi)\vert \leq  \bigg| \hat{\mu}_{\varepsilon}(\varphi)-\frac{\varepsilon^d \omega_d}{C^d} \mu_{\varepsilon}(\varphi) \bigg| + 
 \bigg| \frac{\varepsilon^d \omega_d}{C^d} \mu_{\varepsilon}(\varphi)- \rho \mathcal{L}^d(\varphi)\bigg| 
\\
= \bigg| \sum_{x\in\spt \mu_\e} \frac{1}{C^d}\int_{B_\e(x)} \f(y) -\f(x) \, dy  \bigg| + 
 \bigg| \frac{\varepsilon^d \omega_d}{C^d} \mu_{\varepsilon}(\varphi)- \rho \mathcal{L}^d(\varphi)\bigg| 
 \\
\le \sum_{x\in\spt \mu_\e} \frac{1}{C^d}\int_{B_\e(x)} |\f(y) -\f(x)| \, dy  + 
 \bigg| \frac{\varepsilon^d \omega_d}{C^d} \mu_{\varepsilon}(\varphi)- \rho \mathcal{L}^d(\varphi)\bigg| 
 \\
  \leq 2\e \frac{\varepsilon^d \omega_d}{C^d}\mu_{\varepsilon}(\mathbb{R}^d)  \|\nabla \f\|_{L^\infty}+  
 \bigg| \frac{\varepsilon^d \omega_d}{C^d} \mu_{\varepsilon}(\varphi)- \rho \mathcal{L}^d(\varphi)\bigg|.
  \end{split} 
\end{equation*}
Since $\frac{\varepsilon^d \omega_d}{C^d}\mu_{\varepsilon} \weakstar \rho \mathcal{L}^d$, the claim follows.

\end{proof}

\section{Riesz interactions for $\sigma\in(-d,0)$}
Here we introduce and analyze the Riesz interaction functionals in the integrable case $\sigma\in(-d,0)$.

 \subsection{The energy functionals}  For every  $\varepsilon>0$ and $ \sigma \in (-d,0)$,  let \mbox{  $f_{\varepsilon}^{\sigma}: [0,+\infty) \rightarrow \overline{\mathbb{R}} $}  be defined by

\begin{equation}
 f_{\varepsilon}^{\sigma}(r):=
\left\{
\begin{aligned}
& +\infty & \text{ for } r \in [0,2 \varepsilon) \, , \\
& -\frac{1}{r^{d+\sigma}} & \text{ for }  r \in [2 \varepsilon,\infty) \, . 
\end{aligned}
\right.
\end{equation}

Let $C^d$ be the volume density of the optimal ball  packing in $\mathbb{R}^{d}$ defined in \eqref{C^d0}. 
 
Let $X= \{x_{1}, \cdots, x_{N} \}$ be a finite subset of $\mathbb{R}^{d}$. The corresponding  energy  $F_{\varepsilon}^{\sigma}(X)$ is defined  as
\begin{equation*}
F_{\varepsilon}^{\sigma}(X):= 
\sum_{i\neq j} f_{\varepsilon}^{\sigma}(\vert x_{i} -x_{j} \vert) \bigg( \frac{\omega_d \varepsilon^{d}}{C_{d}} \bigg)^{2}.
\end{equation*}
Clearly, there is a one-to-one correspondence, that we denote by $\mathcal{A}$,  between the family of finite subsets of $\mathbb{R}^{d}$ and the family of {empirical measures}. 
We introduce the energy   \mbox{$\mathcal{F}_{\varepsilon}^{\sigma} : \mathcal{M}_{b}(\mathbb{R}^{d}) \rightarrow \overline{\mathbb{R}} $ } as a function of the empirical measure as
follows:
\begin{equation}\label{deffes}
\mathcal{F}_{\varepsilon}^{\sigma}(\mu):=
\left\{
\begin{aligned}
& F_{\varepsilon}^{\sigma}(\mathcal{A}( \mu) ) & \text{if $ \mu  \in \mathcal{EM}_\e$,} \\
& +\infty & \text{elsewhere.} 
\end{aligned}
\right.
\end{equation} 
The functional $\mathcal{F}_{\varepsilon}^{\sigma}$ may  also be rewritten as 
\begin{equation*}
\mathcal{F}_{\varepsilon}^{\sigma}(\mu)=
\left\{
\begin{aligned}
& \int_{\mathbb{R}^{d}} \int_{\mathbb{R}^{d}} f_{\varepsilon}^{\sigma}(\vert x-y \vert) \bigg( \frac{\varepsilon^{d} \omega_d}{C_{d}}\bigg)^{2}d \mu \otimes \mu  & \text{if $ \mu  \in \mathcal{EM}_\e$,} \\
& +\infty & \text{elsewhere.} 
\end{aligned}
\right.
\end{equation*} 

We observe that the range of the  functionals $\mathcal{F}_{\varepsilon}^{\sigma}$ is $(-\infty,0] \cup \{+\infty\}$.
Therefore, we do not expect compactness properties for sequences with bounded energy. In fact, it is easy to construct, adding more and more masses,  a sequence $\{\mu_\e\}_{\e\in(0,1)}\subset \mathcal{EM}_\e$
with $\e^d \mu_\e(\R^d)\to +\infty$ and $\mathcal{F}_{\varepsilon}^{\sigma} (\mu_\e) \to -\infty$   as $\e\to 0$.  Moreover, tight convergence can also fail by  loss of mass at infinity, 
also for sequences with  
$\e^d \mu_\e(\R^d)\le C$.  Indeed, let   ${T}^{d}$ be an optimal configuration for the optimal packing, as  in Definition \ref{def.T^d e C^d}.
Let $\{z_\e\}_\e \subset \R^d$ with $|z_\e| \to +\infty$ as $\e\to 0$.   Setting $\mu_{\varepsilon}=\sum_{x\in \e \mathrm{T}^{d} \cap B_1(z_\e)}  \delta_{x}$, 
we have that $\e^d \mu_\e(\R^d)\le C$ for some $C$ independent of $\e$, but in general $\e^d \mu_\e$ does not admit converging subsequences in the tight topology. 

Now we perturb the energy functionals by adding suitable confining forcing terms that yield the desired compactness properties.  	

Let $g \in C^0(\mathbb{R}^d)$.
Recalling that $C^d$ is the volume density defined in \eqref{C^d0},  for all $\e \in(0,1)$  we introduce the functionals  $\mathcal{T}_{\varepsilon}^{\sigma}: \mathcal{M}_{b}(\mathbb{R}^d) \rightarrow \overline{\mathbb{R}}$ 
defined as
\begin{equation} \label{Funzmoddiscret}
\mathcal{T}_{\varepsilon}^{\sigma}(\mu):=  \mathcal{F}_{\varepsilon}^{\sigma}(\mu)+\mathcal{G}_{\varepsilon}^{\sigma}(\mu),
\end{equation}
where
\begin{equation*}
\mathcal{G}_{\varepsilon}^{\sigma}(\mu):= \int_{\mathbb{R}^d} g(x) \frac{\varepsilon^d \omega_d}{C^d} d\mu \, .
\end{equation*} 

\subsection{Compactness} 
In this section we study  compactness properties for the functionals $\mathcal{T}_{\varepsilon}^{\sigma}$ introduced in \eqref{Funzmoddiscret}. 
We assume that
\begin{equation}\label{defg}
g(x) \ge C_1 + C_2  |x|^{-\sigma}, \qquad \text{ for some } C_1\in\R,\,C_2> 0.
\end{equation}

\begin{theorem}[Compactness for ${\mathcal{T}}_{\epsilon}^{\sigma}$]\label{teco}
There exists a constant $C^*(\sigma,d) >0$ such that, if $g$ satisfies \eqref{defg} with $C_2> C^*(\sigma,d)$, then the following compactness property hold:
let $M>0$ and let 
 $\{ \mu_{\varepsilon}\}_{\varepsilon \in (0,1)} \subset \mathcal{M}_{b}(\mathbb{R}^d)$ be such that 
\begin{equation*}
 {\mathcal{T}}_{\varepsilon}^{\sigma}(\mu_{\varepsilon}) \leq M, \quad \text{ for all } \varepsilon>0 \, .
\end{equation*}
Then, up to a subsequence, $\frac{\varepsilon^d \omega_d}{C^d} \mu_{\varepsilon} \rightarrow \rho \mathcal{L}^d$ tightly in $\mathcal{M}_{b}(\mathbb{R}^d)$, for some $\rho \in \mathrm{L}^1(\mathbb{R}^d, [0,1])$.	
\end{theorem}
\begin{proof}

In view of \eqref{defg},  it is enough to prove the theorem for $g(x) = C_1 +  C_2 |x|^{-\sigma}$ with  $C_2>C^*(\sigma,d)$ for some $C^*(\sigma,d) >0$.
We divide the proof in several steps. 

\vskip3pt
\noindent 
{\it Step 1. }
For all $\mu \in \mathcal{EM}_{\varepsilon}  $  set $K_{\varepsilon}(\mu):= \e^d \omega_d \mu(\R^d)$ and  let $R_{\varepsilon}(\mu)>0$ be such that $R_{\varepsilon}(\mu)^d \omega_{d}=  K_{\varepsilon}(\mu) $.
In this step we  prove that there
exists $\tilde C(\sigma,d)>0$ such that for all $\mu:=\sum_{i=1}^N \delta_{x_i} \in \mathcal{EM}_{\varepsilon}(\mathbb{R}^d)$ we have
$$
 \sum_{i=1}^{N} \frac{\varepsilon^d \omega_d}{C^d} \vert x_i \vert^{-\sigma} \geq \tilde C(\sigma,d)  \big(K_{\varepsilon}(\mu)\big)^{\frac{d-\sigma}{d}} \, .
 $$
Here and later on we will assume without loss of generality (and whenever it will be convenient) that $|x_i| \ge \e$ for all $x_i\in {\spt}(\mu)$. By triangular inequality we have  $|y|\le |x_i|+\e \le 2 |x_i|$ for all $y\in B_\e(x_i)$. Then, 
$$
\omega_d \e^d |x_i|^{-\sigma} = \int_{B_\e(x_i)} |x_i|^{-\sigma} \, dy \ge 
\frac{1}{2^{-\sigma}} \int_{B_\e(x_i)} |y|^{-\sigma} \, dy.
$$
Let $A_\e$ be the union of all the balls $B_\e(x_i)$. We have 
\begin{multline*}
\sum_{i=1}^{N} \frac{\varepsilon^d \omega_d}{C^d} \vert x_i \vert^{-\sigma} 
\ge
\sum_{i=1}^N 
\frac{1}{2^{-\sigma} C^d} \int_{B_\e(x_i)} |y|^{-\sigma} \, dy
\\
=
\frac{1}{2^{-\sigma} C^d} \int_{A_\e} |y|^{-\sigma} \, dy
=
\frac{1}{2^{-\sigma} C^d} \int_{A_\e\cap B_{R_{\varepsilon}(\mu)} } |y|^{-\sigma} \, dy
+
\frac{1}{2^{-\sigma} C^d} \int_{A_\e\setminus B_{R_{\varepsilon}(\mu)} } |y|^{-\sigma} \, dy 
\\
\ge
\frac{1}{2^{-\sigma} C^d} \int_{B_{R_{\varepsilon}(\mu)}} |y|^{-\sigma} \, dy 
=
\tilde C(\sigma,d)  \big(K_{\varepsilon}(\mu)\big)^{1-\frac{\sigma}{d}} \, ,
\end{multline*}
where in the last inequality we have used that  $|A_\e|= K_{\varepsilon}(\mu) = |B_{R_{\varepsilon}(\mu)}|$, and that $|y_1|^{-\sigma}\ge |y_2|^{-\sigma}$ for all $y_1 \in  
A_\e\setminus B_{R_{\varepsilon}(\mu)} $, $y_2\in B_{R_{\varepsilon}(\mu)}$.

\vskip3pt
\noindent 
{\it Step 2. }
Here we prove that there exists $\hat C(\sigma,d)>0$ such that, for all $\mu \in \mathcal{EM}_{\varepsilon}$, 
   \begin{equation*}
	\frac{1}{(C^d)^2}   
	\sum_{i \neq j} \frac{( {\varepsilon^d \omega_d})^{2}}{\vert x_i -x_j \vert^{d +\sigma}}  \leq   
		  \hat C(\sigma,d)\big( K_\varepsilon(\mu) \big)^{1-\frac{\sigma}{d}}.
	\end{equation*}
First, we observe that by triangular inequality $\vert x_i - x_j \vert \geq \frac{1}{3} \vert x- y \vert $ for all $ (x,y) \in B_{\varepsilon}(x_i) \times B_{\varepsilon}(x_j)$. Then, 
there exists $\hat C(\sigma,d) >0$ such that 
\begin{multline}\label{Riesz}
 \frac{1}{(C^d)^2}  
 \sum_{i \neq j} \frac{( {\varepsilon^d \omega_d})^{2}}{\vert x_i -x_j \vert^{d+\sigma}} \le
\hat C(\sigma,d)  
\sum_{i \neq j} \int_{B_{\varepsilon}(x_i)} \int_{B_{\varepsilon}(x_j)} \frac{1}{\vert x-y \vert^{d+\sigma}}dxdy  \\
\le 
\hat C(\sigma,d) \int_{B_{R_{\varepsilon}(\mu_\e)}(0)} \int_{B_{R_{\varepsilon}(\mu_\e)}(0)} \frac{1}{\vert x-y \vert^{d+ \sigma}}dxdy
\\
 \leq 
\hat C(\sigma,d)   \int_{B_{R_{\varepsilon}(\mu_\e)}(0)} dx \int_{B_{2 R_{\varepsilon}(\mu_\e)}(0)} \frac{1}{\vert z \vert^{d+\sigma}} dz= \hat C(\sigma,d) (K_{\varepsilon}(\mu))^{1-\frac{\sigma}{d}},
\end{multline}
where the second inequality is nothing but Riesz inequality, see \cite{LiebLoss}.

\vskip3pt
\noindent 
{\it Step 3. }
Here we prove that there exists $C^*(\sigma,d)>0$  such that, if $C_2> C^*(\sigma,d)$, then
the following implication holds:
$$
\text{ if } \limsup_{\e} \mathcal{T}_\varepsilon^{\sigma}(\mu_\varepsilon) < +\infty,  \text{ then }
\limsup_{\e} \frac{\varepsilon^d \omega_d}{C^d} \mu_{\varepsilon}(\mathbb{R}^d) < +\infty.
$$ 
By {\it Step 1} and {\it Step 2} there exists $C(\sigma,d)>0$ such that 
$$
 \mathcal{T}_\varepsilon^{\sigma}(\mu_\varepsilon) \ge 
\frac{C_1}{C^d}  K_{\varepsilon}(\mu)  +
 (-\hat C(\sigma,d) + C_2 \tilde C(\sigma,d)) (K_{\varepsilon}(\mu))^{1-\frac{\sigma}{d}} .
$$
It is then sufficient to choose $C_2$ large enough, so that $ (-\hat C(\sigma,d) + C_2 \tilde C(\sigma,d)) >0$.

\vskip3pt
\noindent 
{\it Step 4. }
We now prove the tight converge, up to a subsequence, of sequences $\{\mu_\e\}_\e$ with bounded energy. In view of Lemma \ref{dime}, this step concludes the proof of the theorem.  By {\it Step 3} we have that $ \frac{\varepsilon^d \omega_{d}}{C^d}\mu_{\varepsilon}(\mathbb{R}^d) \leq \tilde{M}$ for all $\varepsilon \in (0,1)$ and some $\tilde M>0$.  
Arguing by  contradiction, assume that  there exists $\delta>0$, $\varepsilon_{n} \rightarrow 0^+$ and $R_{n} \rightarrow +\infty$ as $n \rightarrow +\infty$, such that 
\begin{equation}
\frac{\varepsilon_{n}^{d} \omega_d}{C^d} \mu_{\varepsilon_{n}}(\mathbb{R}^d\setminus B_{R_n}(0)) \geq \delta \quad \forall n.
\end{equation}
Now let us split $\mu_{\varepsilon_{n}}$ into two components: $\mu_{\e_n}^1:= \mu_{\varepsilon_{n}}\lfloor_{B_{R_n}(0)}$ and $\mu_{\e_n}^2:= \mu_{\varepsilon_{n}}\lfloor_{\mathbb{R}^d \setminus B_{R_n}(0)}$; then 
\begin{equation}\label{prima}
\begin{split}
\mathcal{T}_{\varepsilon_{n}}^{\sigma}(\mu_{\varepsilon_{n}})= 
& \mathcal{T}_{\varepsilon_{n}}^{\sigma}(\mu_{\e_n}^1) 
 +\mathcal{T}_{\varepsilon_{n}}^{\sigma}(\mu_{\e_n}^2) \\
& -2 \int_{B_{R_n}(0)} \int_{\mathbb{R}^d \setminus B_{R_n}(0)} \frac{1}{\vert x-y \vert^{\sigma+d}}\bigg( \frac{\varepsilon_{n}^{d} \omega_d}{C^d}\bigg)^{2} d \mu_{\varepsilon_{n}} \otimes \mu_{\varepsilon_{n}}. 
\end{split}
\end{equation}
From  {\it Step 2 }  we have that there exists $C>0$ independent of $n$ such that
\begin{equation}\label{eq2-teo-comp-tau}
\mathcal{T}_{\varepsilon_{n}}^{\sigma}(\mu_{\varepsilon_{n}}^1)
 \ge -\hat C(\sigma,d) (K_{\varepsilon_n}(\mu_{\e_n}^1))^{1-\frac{\sigma}{d}} \ge C.
 \end{equation}
 Again by {\it Step 2}, applied now to $\mu_{\e_n}^2$, we have that there exists $C>0$ independent of $n$ such that 
\begin{equation*}
\int_{\mathbb{R}^d \setminus B_{R_n}(0)} \int_{\mathbb{R}^d \setminus B_{R_n}(0)}\frac{1}{\vert x-y \vert^{d+\sigma}}\bigg(\frac{\varepsilon_n^d \omega_{d}}{C^d}\bigg)^{2}d \mu_{\varepsilon_n} \otimes \mu_{\varepsilon_n} \leq  C.
\end{equation*} 
Therefore, 
\begin{equation}\label{eq3-teo-comp-tau}
\mathcal{T}_{\varepsilon_{n}}(\mu_{\varepsilon_{n}}^2) \geq -C  -|C_1| \tilde M + C_2 \int_{\mathbb{R}^d \setminus B_{R_n}(0)} R_{n}^{-\sigma} \frac{\varepsilon_{n}^d \omega_d}{C^d}d \mu_{\varepsilon_{n}}  
\geq  -C +C_2\delta R_{n}^{-\sigma} .
\end{equation}
Finally, by Riesz inequality (or equivalently, arguing as in \eqref{Riesz}) we have that there exists $C>0$ independent of $n$ such that 
\begin{equation}\label{eq4-teo-comp-tau}
 \int_{B_{R_n}(0)} \int_{\mathbb{R}^{d}\setminus B_{R_n}(0)} \frac{-1}{\vert x-y \vert^{d+\sigma}} \bigg(\frac{\varepsilon_{n}^d \omega_d}{C^d}\bigg)^2 d \mu_{\varepsilon_{n}}\otimes\mu_{\varepsilon_{n}}(x,y)\geq -C
\end{equation}
Now plugging  \eqref{eq2-teo-comp-tau},\eqref{eq3-teo-comp-tau} and \eqref{eq4-teo-comp-tau}    into  \eqref{prima}, we deduce that 
\begin{equation*}
 M\geq \mathcal{T}_{\varepsilon_{n}}^{\sigma}(\mu_{\varepsilon_{n}})\geq -C + C_2 \delta R_{n}^{-\sigma},
\end{equation*}
for some $C$ independent of $n$, which clearly provides  a contradiction for $n $ large enough.

\end{proof}

\subsection{$\Gamma$-convergence}
In this section we study the  $\Gamma$-convergence of the energy functionals defined in \eqref{deffes} and \eqref{Funzmoddiscret}.

\begin{proposition}\label{propco}
	Let $\{ \mu_{\varepsilon}\}_{\varepsilon \in (0,1)} \subset \mathcal{M}_{b}(\mathbb{R}^d)$ with $\mu_{\varepsilon} \in \mathcal{EM}_{\varepsilon}$ for all $\varepsilon \in (0,1)$
	and let $\rho\in L^1(\R^d;[0,1]) $ be such that $\frac{\varepsilon^d \omega_d}{C^d} \mu_{\varepsilon} \rightarrow \rho \mathcal{L}^d$ tightly.
	Let moreover
	$h(x,y):=\frac{1}{\vert x-y\vert^{d+\sigma}}$ for all $x, \, y\in \R^d$ with  $x \neq y$. Then, 
	\begin{equation*}
	\biggl(\frac{\varepsilon^d \omega_d}{C^d}\biggr)^2 \mu_{\varepsilon} \otimes \mu_{\varepsilon}(h) \rightarrow \rho \mathcal{L}^d \otimes \rho \mathcal{L}^d(h), \quad \text{ as } \; \varepsilon \rightarrow 0^+ \, .
	\end{equation*}
\end{proposition}
\begin{proof}
The proof is divided in several steps:

{\it Step 1. } Here we prove that 
	\begin{equation*}
	\hat{\mu}_{\varepsilon}\otimes\hat{\mu}_{\varepsilon}(h)\rightarrow \rho \mathcal{L}^d \otimes \rho \mathcal{L}^d(h), \; \; \text{ as } \; \varepsilon \rightarrow 0^+,
	\end{equation*}
where $\hat{\mu}_{\varepsilon}$ are defined as in \eqref{defhm} (with $\mu$ replaced by $\mu_\e$). 

For all $R>0$ we set
\begin{equation}\label{DR}
D(R):= \underset{x \in \mathbb{R}^d}{\bigcup} (\{x\} \times B_R(x)).
\end{equation}
We have
\begin{align}
& \bigg| \int_{\mathbb{R}^d} \int_{\mathbb{R}^d} \frac{1}{\vert x-y \vert^{d +\sigma}} d \hat{\mu}_{\varepsilon} \otimes \hat{\mu}_{\varepsilon}- \int_{\mathbb{R}^d} \int_{\mathbb{R}^{d}} \frac{1}{\vert x-y \vert^{d+ \sigma}} \rho(x) \rho(y)dxdy \bigg| 
\nonumber \\
\leq  & \int_{D(R)} \frac{1}{\vert x-y \vert^{d+\sigma}}d \hat{\mu}_{\varepsilon} \otimes \hat{\mu}_{\varepsilon}   \label{ambarabacci1} \\
 & + \int_{D(R)} \frac{1}{\vert x-y \vert^{d+\sigma}} \rho(x)\rho(y)dxdy   \label{ambarabacci2} \\
& + \bigg| \int_{\mathbb{R}^{2d} \setminus D(R)} \frac{1}{\vert x-y \vert^{d +\sigma}} d \hat{\mu}_{\varepsilon} \otimes \hat{\mu}_{\varepsilon}- \int_{\mathbb{R}^{2d} \setminus D(R)} \frac{1}{\vert x-y \vert^{d+\sigma}} \rho(x)\rho(y)dxdy\bigg|.  \label{ambarabacci3}
\end{align}
Moreover, we have
\begin{align*}
\int_{D(R)} \frac{1}{\vert x-y \vert^{d+\sigma}}d \hat{\mu}_{\varepsilon} \otimes \hat{\mu}_{\varepsilon} =
\int_{\R^d}  d\hat{\mu}_{\varepsilon}  \int_{B_R(x)}   \frac{1}{\vert x-y \vert^{d+\sigma}} d\hat{\mu}_{\varepsilon} \\
\le
\int_{\R^d}  d\hat{\mu}_{\varepsilon}   \frac{1}{C^d} \int_{B_R(x)}   \frac{1}{\vert x-y \vert^{d+\sigma}}\, dy 
= \hat{\mu}_{\varepsilon}(\R^d) \omega(R),
\end{align*}
where $\omega(R) \to 0$ as $R\to 0$. 
This proves that the quantity in \eqref{ambarabacci1} tends to $0$ as $R\to 0$, uniformly in $\e$; a fully analogous argument shows that the same  holds true also for the quantity in \eqref{ambarabacci2}. Finally, the quantity in \eqref{ambarabacci3} tends to $0$ as $\e\to 0$ (for fixed $R$)  since   
 $ \frac{1}{\vert x-y \vert^{d+\sigma}}$ is continuous and bounded in $\R^{2d} \setminus D(R)$, 
and $\hat \mu_e \otimes \hat\mu_\e \to \rho \mathcal L^d \otimes \rho \mathcal L^d$ tightly in $\R^{2d}$,  and  hence also in $\R^{2d} \setminus D(R)$.

{\it Step 2. } Here we prove that 
	\begin{equation*}
	 \bigg(\frac{\varepsilon^d \omega_d}{C^d}\bigg)^2 \mu_{\varepsilon} \otimes \mu_{\varepsilon}(h)-
	 \hat{\mu}_{\varepsilon}\otimes \hat{\mu}_{\varepsilon}(h)   \rightarrow 0 \quad \varepsilon \rightarrow 0^+.
	\end{equation*}
Let $x_i,x_j \in \spt(\mu_{\varepsilon})$, with $i \neq j$;  for all $ x \in B_\e(x_i), \, y \in B_\e(x_j)$, by triangular inequality we have 
$\vert x-y \vert \leq 2 \vert x_i -x_j \vert$, and hence  
\begin{equation} \label{grigiotto}
\bigg(\frac{\varepsilon^d \omega_d}{C^d}\bigg)^2 \frac{1}{\vert x_i -x_j \vert^{d +\sigma}} \leq 2^{d+\sigma} \int_{B_\varepsilon(x_i)} \int_{B_{\varepsilon}(x_j)}\frac{1}{(C^d)^2} \frac{1}{\vert x-y \vert^{d +\sigma}}dxdy .
\end{equation}
 Let $D(R)$ be the set defined in \eqref{DR}.  We obtain that 
\begin{align}
& \bigg| \int_{\mathbb{R}^{2d}} \frac{1}{\vert x-y \vert^{d+\sigma}} \bigg(\frac{\varepsilon^d \omega_d}{C^d}\bigg)^2 d \mu_{\varepsilon} \otimes \mu_{\varepsilon} - \int_{\mathbb{R}^{2d}}  \frac{1}{\vert x-y \vert^{d +\sigma}} d \hat{\mu}_{\varepsilon} \otimes \hat{\mu}_{\varepsilon}\bigg|  \\
\leq & \bigg| \int_{D(R)} \frac{1}{\vert x-y \vert^{d +\sigma}} \biggl(\frac{\varepsilon^d \omega_d}{C^d}\biggr)^2 d \mu_{\varepsilon} \otimes \mu_{\varepsilon} \bigg| \label{grigiotto1} \\ 
& + \bigg| \int_{D(R)} \frac{1}{\vert x-y \vert^{d + \sigma}} d \hat{\mu}_{\varepsilon} \otimes \hat{\mu}_{\varepsilon} \bigg| \label{grigiotto2}\\
& + \bigg| \int_{\mathbb{R}^{2d} \setminus D(R)} \frac{1}{\vert x-y \vert^{d +\sigma}} \bigg(\frac{\varepsilon^d \omega_d}{C^d}\bigg)^2 d \mu_{\varepsilon} \otimes \mu_{\varepsilon} 
\label{grigiotto3}
\\
\nonumber  &  \phantom{\qquad \qquad} - \int_{\mathbb{R}^{2d} \setminus D(R)} \frac{1}{\vert x-y \vert^{d +\sigma}} d \hat{\mu}_{\varepsilon} \otimes \hat{\mu}_{\varepsilon} \bigg|. 
\end{align}
By  \eqref{grigiotto} we deduce that the quantiy in  \eqref{grigiotto1} is, up to a prefactor,  less than or equal to the quantity in \eqref{grigiotto2}, which, as proved in {\it Step 1}, tends to zero as 
$R\to 0$, uniformly with respect to $\e$. 
Finally, since  $ \frac{1}{\vert x-y \vert^{d+\sigma}}$ is continuous and bounded in $\R^{2d} \setminus D(R)$,  by Lemma \ref{lediti} we easily deduce that, for  any  fixed $R>0$,  the quantity in \eqref{grigiotto3} 
 tends to zero as  $\e\to 0$.  This concludes the proof of {\it Step 2.}
 
 The proof of the claim is clearly a consequence of {\it Step 1} and {\it Step 2}.

\end{proof}

We now introduce the candidate $\Gamma$-limit $\mathcal{F}^{\sigma}: \mathcal{M}_{b}(\mathbb{R}^d) \rightarrow \mathbb{R} \cup \{ +\infty\}$ 
defined by 
\begin{equation*}
\mathcal{F}^{\sigma}(\mu):=
\left\{
\begin{aligned}
& \int_{\mathbb{R}^d} \int_{\mathbb{R}^d}-\frac{1}{\vert x-y \vert^{d+\sigma}} d \mu \otimes \mu   & \text{if $\mu \leq \mathcal{L}^d$,} \\
& + \infty  & \text{elsewhere.} 
\end{aligned}
\right.
\end{equation*}

\begin{theorem}\label{Gammaconvprimocaso}
	Let $ \sigma \in (-d,0)$. The following $\Gamma$-convergence result holds true.
	\begin{enumerate}
		\item ($\Gamma$-liminf inequality) For every $\rho \in \mathrm{L}^1(\mathbb{R}^d,[0,1])$ and for every sequence $\{ \mu_{\varepsilon} \}_{\varepsilon \in (0,1)} \subset \mathcal{M}_{b}(\mathbb{R}^d)$ with $\frac{\varepsilon^d \omega_d}{C^d} \mu_{\varepsilon} \rightarrow \rho \mathcal{L}^d$ tightly in $\mathcal{M}_b(\mathbb{R}^d)$ it holds
		\begin{equation*}
		\mathcal{F}^{\sigma}(\rho\mathcal{L}^d) \leq \liminf_{\varepsilon \rightarrow 0^+} \mathcal{F}_{\varepsilon}^{\sigma}(\mu_{\varepsilon}).
		\end{equation*}
		\item ($\Gamma$-limsup inequality) For every $\rho \in \mathrm{L}^1(\mathbb{R}^d,[0,1])$, there exists a sequence $\{\mu_{\varepsilon}\}_{\varepsilon \in (0,1)} \subset \mathcal{M}_{b}(\mathbb{R}^d)$ such that $\frac{\varepsilon^d \omega_d}{C^d} \mu_{\varepsilon} \rightarrow \rho \mathcal{L}^d$ tightly in $\mathcal{M}_{b}(\mathbb{R}^d)$ and
		\begin{equation*}
		\mathcal{F}^{\sigma}(\rho \mathcal{L}^d) \geq \limsup_{\varepsilon \rightarrow 0^+} \mathcal{F}_{\varepsilon}^{\sigma}(\mu_{\varepsilon}). 
		\end{equation*}
	\end{enumerate}
\end{theorem}
\begin{proof}
The  $\Gamma$-liminf inequality is a direct consequence of Proposition \ref{propco} 
while  the $\Gamma$-limsup inequaility is a direct consequence of Lemma \ref{LemmacostruzF} and again of Proposition \ref{propco}.
\end{proof}

Now we introduce the $\Gamma$-limit 
$\mathcal{T}^{\sigma}: \mathcal{M}_{b}(\mathbb{R}^d) \rightarrow \overline{\mathbb{R}} $
of the functionals $\mathcal{T}_{\varepsilon}^{\sigma}$
introduced in \eqref{Funzmoddiscret},
 defined by
\begin{equation*}
\mathcal{T}^{\sigma}(\mu):=
\left\{
\begin{aligned}
& \mathcal{F}^{\sigma}(\mu)+ \int_{\mathbb{R}^d} g(x) d \mu(x)  & \text{if $\mu \leq \mathcal{L}^d$,} \\
& + \infty  & \text{elsewhere.} 
\end{aligned}
\right.
\end{equation*}

\begin{theorem}\label{gcsc}
	Let $\sigma \in (-d,0)$, let $g\in\mathcal C^0(\R^d)$ satisfying   $g(x)\ge 0 $
	for $|x|$ large enough, and let $\mathcal T^\sigma_\e$ be defined in \eqref{Funzmoddiscret}. The following $\Gamma$-convergence result holds true.
	\begin{enumerate}
		\item ($\Gamma$-liminf inequality) For every $\rho \in \mathrm{L}^1(\mathbb{R}^d,[0,1])$ and for every sequence $\{\mu_{\varepsilon}\}_{\varepsilon \in (0,1)} \subset \mathcal{M}_{b}(\mathbb{R}^d)$ with $\frac{\varepsilon^d \omega_d}{C^d}\mu_{\varepsilon} \rightarrow \rho \mathcal{L}^d$ tightly in $\mathcal{M}_{b}(\mathbb{R}^d)$ it holds 
		\begin{equation*}
		\mathcal{T}^{\sigma}(\rho \mathcal{L}^d) \leq \liminf_{\varepsilon \rightarrow 0^+} \mathcal{T}_{\varepsilon}^{\sigma}(\mu_{\varepsilon}). 
		\end{equation*}
		\item ($\Gamma$-limsup inequality) For every $\rho \in \mathrm{L}^1(\mathbb{R}^d,[0,1])$ there exists a sequence $\{\mu_{\varepsilon}\}_{\varepsilon \in (0,1)}$ such that $ \frac{\varepsilon^d \omega_d}{C^d}\mu_{\varepsilon} \rightarrow \rho \mathcal{L}^d$ tightly in $\mathcal{M}_{b}(\mathbb{R}^d)$ and
		\begin{equation*}
		\mathcal{T}^{\sigma}(\rho\mathcal{L}^d) \geq \limsup_{\varepsilon \rightarrow 0^+} \mathcal{T}_{\varepsilon}^{\sigma}(\mu_{\varepsilon}).
		\end{equation*}
		\end{enumerate}
\end{theorem}
\begin{proof}
We start by proving (1). It is easy to prove (see \cite[Proposition 1.62]{AmbFuscPall})  that the term 
$\int_{\R^d} g(x) d\mu$ is lower semicontinuous 
with respect to tight convergence. Then, 
by Theorem \ref{Gammaconvprimocaso} we obtain that 
\begin{equation*}
\begin{split}
& \mathcal{T}^{\sigma}(\rho \mathcal{L}^d) =  \mathcal{F}^{\sigma}(\rho \mathcal{L}^d)+ \int_{\mathbb{R}^d} g(x) \rho(x) dx  \\
\leq & \liminf_{\varepsilon \rightarrow 0^+} \mathcal{F}_{\varepsilon}^{\sigma}(\mu_{\varepsilon})+ \liminf_{\varepsilon \rightarrow 0^+}  \int_{\mathbb{R}^d} g(x) \frac{\varepsilon^d \omega_d}{C^d} d \mu_{\varepsilon}(x)  \\
\leq & \liminf_{\varepsilon \rightarrow 0^+} \Big( \mathcal{F}_{\varepsilon}^{\sigma}(\mu_{\varepsilon})+  \int_{\mathbb{R}^d} g(x) \frac{\varepsilon^d \omega_d}{C^d} d \mu_{\varepsilon}(x)\Big) = \liminf_{\varepsilon \rightarrow 0^+} \mathcal{T}_{\varepsilon}^{\sigma}(\mu_{\varepsilon}).
\end{split}
\end{equation*}

We now prove (2). 
First consider the case $\rho\in C^0_c(\R^d)$. 
Let $R>0 $ be such that $\spt (\rho) \subset B_R$  and let $\{\mu_\e\}$ be  the recovery sequence provided by Theorem \ref{Gammaconvprimocaso};
then, it is easy to see that $\{\mu_\e \chi_{B_R}\}$ provides a recovery sequence  also
for the functionals  $\mathcal{T}^{\sigma}_\e$.
The general case follows by a standard diagonalization argument. 
Indeed, 
for any sequence $\{\f_n\}\subset C^0(\R^d;[0,1])$ converging to $\f$ in $L^1$ we have 
$\mathcal{F}^{\sigma}(\f_n\mathcal{L}^d)\to
\mathcal{F}^{\sigma}(\f\mathcal{L}^d)$ (see for instance the proof of Proposition \ref{propco}). Then,
for any sequence $\{\rho_n\}\subset C^0_c(B_R;[0,1])$ converging to $\rho\chi_{B_r}$ in $L^1$ we have 
$\mathcal{T}^{\sigma}(\rho_n\mathcal{L}^d)\to
\mathcal{T}^{\sigma}(\rho \chi_{B_R}\mathcal{L}^d)$.
Moreover, since $\rho$ is nonnegative and $g(x)$ is positive for $|x|$ large enough, we have that
$$
\int_{\mathbb{R}^d} g(x)  \rho(x) \chi_{B_R} (x) dx
\to 
\int_{\mathbb{R}^d} g(x)  \rho(x)  (x) dx
$$ 
as $R\to +\infty$. We deduce that 
$\mathcal{T}^{\sigma}(\rho\chi_{B_R}\mathcal{L}^d)\to \mathcal{T}^{\sigma}(\rho\mathcal{L}^d)$
as $R\to +\infty$.
Therefore, there exists 
a sequence $\{\rho_m\}_{m\in\N} \subset  \mathcal C^0_c(\R^d)$ such that $\rho_m\to \rho$ in $L^1(\R^d)$  and
$\mathcal{T}^{\sigma}(\rho_m\mathcal{L}^d)\to 
\mathcal{T}^{\sigma}(\rho\mathcal{L}^d)$
as $m\to +\infty$.

\end{proof}

\subsection{Asymptotic behaviour of minimizers}

Here we analyze the asymptotic behaviour of minimizers of the functionals ${\mathcal{T}}_{\e}^{\sigma}$
defined in \eqref{Funzmoddiscret}.

\begin{proposition}[First variation]\label{pfv}
	Let $\rho  \mathcal{L}^d$ be a minimizer of ${\mathcal{T}}^{\sigma}$.
	For almost every $x\in\R^d$ such that  $0<\rho(x)<1$
	we have 
	$$
	 g(x)  -  2 \int_{\mathbb{R}^{d}}  \frac{1}{\vert x-y \vert^{d+\sigma}}    \rho(y) \, dy  =0 \, .
	$$
\end{proposition}

\begin{proof}
	Let $h(x,y):= \vert x-y \vert^{-d-\sigma}$.
	Let $0<\alpha<\beta<1$ and  set
	$$
	E_{\alpha,\beta}:= \{x\in\R^d: \, \alpha < \rho(x) <\beta\}. 
	$$
		Let $E\subseteq E_{\alpha,\beta}$, and set $u:=\chi_E$. Then, for $\e$ small enough the function $\rho + \e u$ takes values in $(0,1)$. By minimality of $\rho$ we deduce that
	\begin{multline*}
	0\le {\mathcal{T}}^{\sigma}(\rho+\e u) - 	{\mathcal{T}}^{\sigma}(\rho) 
	\\
	=  \e \int_{\R^d}  g(x) u(x) \, dx -2 
	\e \int_{\mathbb{R}^{2d}}  h(x,y)     \rho(y) u(x) \, dy \, dx + o(\e),
	\end{multline*}
where $o(\e)/\e \to 0 $ as $\e\to 0$.  
We deduce that  
	\begin{equation}\label{fv}
	\int_{\R^d}  g(x) u(x) \, dx -2
	 \int_{\mathbb{R}^{2d}}  h(x,y)   \rho(y) u(x) \, dy \, dx = 0 \, .
	\end{equation}
Since the above inequality holds for  $u=\chi_E$ where $E$ is any measurable set contained in $\{x\in\R^d: 0<\rho(x)<1\}$, by the fundamental lemma in the calculus of variations and an easy density argument we deduce the claim.
	
%
%
	 \end{proof}

\begin{theorem}[Behaviour of minimizers]
	Let ${\mathcal{T}}_{\e}^{\sigma}$ be defined in \eqref{Funzmoddiscret} with $g$ satisfying \eqref{defg} for some $C_2> C^*(\sigma,d)$, where 
	$C^*(\sigma,d)$ is the constant provided by Theorem \ref{teco}.
	Let moreover $\mu_\e$ be minimizers of ${\mathcal{T}}_{\e}^{\sigma}$ for all $\e>0$. 
	
	Then, up to a subsequence,  $\frac{\varepsilon^d \omega_d}{C^d} \mu_{\varepsilon} \rightarrow \chi_E  \mathcal{L}^d$ tightly in $\mathcal{M}_{b}(\mathbb{R}^d)$, for some set $E\in \mathcal M_f(\R^d)$. Moreover,
	$ \chi_E  \mathcal{L}^d$ is a minimizer of	${\mathcal{T}}^{\sigma}$.
	Finally, if $g(x):=G(|x|)$ for some increasing function $G:\R^+ \to \R$, then $E$ is a ball.  
\end{theorem}

\begin{proof}
	By Theorem \ref{teco}, up to a subsequence,  $\frac{\varepsilon^d \omega_d}{C^d} \mu_{\varepsilon} \rightarrow \rho  \mathcal{L}^d$ tightly in $\mathcal{M}_{b}(\mathbb{R}^d)$, for some  $\rho\in L^1(\R^d;[0,1])$. Moreover, as a consequence of the $\Gamma$-convergence result established in Theorem \ref{gcsc},  $\rho \mathcal{L}^d$ is a minimizer of ${\mathcal{T}}^{\sigma}$; we have to prove that $\rho$ is a characteristic function.

	Let now $\tilde \rho:= \chi_{\spt(\rho)}$ and let $u:= \tilde \rho - \rho$. 
By \eqref{fv} we have
	\begin{multline*}
		0\le {\mathcal{T}}^{\sigma}(\rho + u) - 	{\mathcal{T}}^{\sigma}(\rho)
		\\
		=
		\int_{\R^d}  g(x) u(x) \, dx -2
		\int_{\mathbb{R}^{2d}}  h(x,y)   \rho(y) u(x) \, dy \, dx - \int_{\R^{2d}} h(x,y) u(y) u(x) \,dy \, dx 
		\\
		=   
		- \int_{\R^{2d}} h(x,y) u(y) u(x) \,dy \, dx \le 0.
		\end{multline*}
We conlcude that the above inequalities are in fact all equalities, which in turns implies $u=0$, i.e., $\tilde \rho = \rho$ and $\rho$ is  a characteristic function.
	
	Finally, if $g$ is radial and increasing with respect to $|x|$, then denoted by $E^*$ the ball centered at $0$ with $|E^*|=|E|$, we have  
	\begin{equation}\label{uniba}
	\mathcal{F}^{\sigma}(E^* \mathcal L^d) <
\mathcal{F}^{\sigma}(E \mathcal L^d), 
\qquad 
\int_{E^*} g(x) d x
\le \int_{E} g(x) d x,	
\end{equation}
where the first (strict) inequality is a consequence of the uniqueness of the ball in the Riesz inequality for characteristic functions interacting through strict increasing potentials (see for instance \cite[Theorem A4]{DeLucaNovPons}).  From \eqref{uniba} we easily conclude that $E$ must be a ball. 
\end{proof}

\section{Riesz interactions for $\sigma\in [0,1)$}
Here we introduce and analyze regularized Riesz interaction functionals in the non-integrable case $\sigma\in [0,1)$.
\subsection{The energy functionals}\label{SezEnergie}

Let $\sigma\in[0,1)$. For every $\e>0$ let $r_\e>0$ be such that $r_\e\to 0$
as $\e\to 0$ and 
\begin{align}\label{condr}
& \frac{\e^{\frac{1}{2\sigma+1}}}{r_\e} \to 0 &\text{ as } \e\to 0 \qquad &\text{ for } \sigma \in (0,1);\\
& \frac{\e |\log (r_\e)|^2 }{r_\e } \to 0 &\text{ as } \e\to 0 \qquad &\text{ for } \sigma =0.
\end{align}
The regularized potentials are defined by
\begin{equation}
 f_{\varepsilon}^{\sigma}(r):=
\begin{cases}
 +\infty   & \text{ for } r \in [0,2 \varepsilon) \, , \\
 0 &  \text{ for } r \in [2 \varepsilon,r_{\varepsilon})\, , \\
 -\frac{1}{r^{d+\sigma}} &  \text{ for } r \in [r_{\varepsilon}, +\infty) \, , 
\end{cases}
\end{equation} 
As in \eqref{deffes}, we introduce the energy functionals
\begin{equation}\label{deffes2}
\mathcal{F}_{\varepsilon}^{\sigma}(\mu):=
\left\{
\begin{aligned}
& F_{\varepsilon}^{\sigma}(\mathcal{A}( \mu) ) & \text{if $ \mu  \in \mathcal{EM}_\e$,} \\
& +\infty & \text{elsewhere.} 
\end{aligned}
\right.
\end{equation}
We will also introduce suitable renormalized energy functionals. To this purpose, 
for all $\sigma \in [0,1)$ and $r \in (0,1]$ we set
\begin{equation}\label{cost,gamma,r,signa}
\gamma_{r}^{\sigma}:=-\int_{B_{1}(0)\setminus B_{r}(0)} \frac{1}{\vert z \vert^{d+\sigma}}dz.
\end{equation}
Notice that
\begin{eqnarray}\label{sgammaR}
\gamma^{\sigma}_r:=\left\{\begin{array}{ll}
\displaystyle d\omega_d\frac{1-r^{-\sigma}}{\sigma}&\textrm{if }\sigma\neq 0\,,\\ 
\displaystyle d\omega_d\log r&\textrm{if }\sigma=0\, .
\end{array}
\right.
\end{eqnarray}
For $\sigma \in [0,1)$ the renormalized energy functionals  $\hat {\mathcal{F}}_{\varepsilon}^{\sigma} : \mathcal{M}_{b}(\mathbb{R}^{d}) \rightarrow \overline \R$ are defined by 
\begin{equation*}
\hat {\mathcal{F}}_{\varepsilon}^{\sigma}(\mu):=
\left\{
\begin{aligned}
& \mathcal F_{\varepsilon}^{\sigma}(\mathcal{A}( \mu) )-\gamma_{r_{\varepsilon}}^{\sigma} \frac{\varepsilon^{d} \omega_d}{C_{d}} \mu(\mathbb{R}^{d}) & \text{if $ \mu  \in \mathcal{EM}_\e$,} \\
& +\infty & \text{elsewhere.} 
\end{aligned}
\right.
\end{equation*} 
The functional $\hat {\mathcal{F}}_{\varepsilon}^{\sigma}$ may be also rewritten as 

\begin{equation*}
\hat {\mathcal{F}}_{\varepsilon}^{\sigma}(\mu) =
\left\{
\begin{aligned}
& \int_{\mathbb{R}^{d}} \int_{\mathbb{R}^{d}} f_{\varepsilon}^{\sigma}(\vert x-y \vert) \bigg( \frac{\varepsilon^{d} \omega_d}{C_{d}}\bigg)^{2}d \mu \otimes \mu-\gamma_{r_\varepsilon}^{\sigma} \frac{\varepsilon^{d} \omega_d}{C_{d}} \mu(\mathbb{R}^{d}) & \text{if $ \mu  \in \mathcal{EM}_\e$,} \\
& +\infty & \text{elsewhere.} 
\end{aligned}
\right.
\end{equation*}

\subsection{The continuous model}
Here we give a short overview of the $\Gamma$-convergence analysis of the continuous model for non-integrable Riesz potentials developed in \cite{DeLucaNovPons}.

First, we introduce the fractional perimeters; for all $\sigma \in (0,1)$, the $\sigma$-fractional perimeter of $E\in\M(\R^d)$ is defined by
\begin{equation*}
P^{\sigma}(E)= \int_{E} \int_{\mathbb{R}^d \setminus E} \frac{1}{\vert x-y \vert^{d+\sigma}}dx dy.
\end{equation*}
For $\sigma=0$, a notion of $0$-fractional perimeter has been introduced in \cite{DeLucaNovPons} as follows.

First, for all $R >1$ 
we set
\begin{equation}\label{cost,gamma,r,signa2}
\gamma_{R}^{0}:=\int_{B_{R}(0)\setminus B_1(0)} \frac{1}{\vert z \vert^{d}}dz.
\end{equation}
Then, the following definition is well posed (namely, the following limit exists, \cite{DeLucaNovPons})
\begin{equation*}
P^0(E):=\lim_{R\to + \infty} \int_{E} \int_{B_R(x)\setminus E} \frac{1}{|x-y|^{d}} \, dx \, dy \, -\gamma^0_R |E|.
\end{equation*}
Now, we introduce the continuous Riesz functionals. 
For all $r\in(0,1)$ let $J_{r}^{\sigma}: {\M}_{f}(\mathbb{R}^d) \rightarrow \overline{\mathbb{R}}$ be the functionals defined by 
\begin{equation*}
J_{r}^{\sigma}(E):= \int_{E} \int_{E \setminus B_{r}(x)} \frac{-1}{\vert x-y \vert^{d+\sigma}}dxdy.
\end{equation*}
The renormalized functionals $\hat{J}_{r}^{\sigma}: {\M}_{f}(\mathbb{R}^d) \rightarrow \overline{\mathbb{R}}$ are defined by
\begin{equation}\label{funzionaliriesznormalizati}
\hat{J}_{r}^{\sigma}(E):= J_{r}^{\sigma}(E) -\gamma_{r}^{\sigma}\vert E \vert,
\end{equation}
where $\gamma_{r}^{\sigma}$ is the constant defined in \eqref{cost,gamma,r,signa}.

Now we introduce the candidate $\Gamma$-limits. 
For $\sigma \in (0,1)$ we define the functional $\hat{\mathcal{F}}^{\sigma}: \mathcal{M}_{b}(\mathbb{R}^d) \rightarrow \overline{\mathbb{R}}$ as
\begin{equation}\label{desi}
\hat{\mathcal{F}}^{\sigma}(\mu):=
\left\{
\begin{aligned}
& P^{\sigma}(E) - \gamma^{\sigma} \vert E \vert  & \text{if $\mu = \chi_{E} \mathcal{L}^d$,} \\
& + \infty  & \text{elsewhere,} 
\end{aligned}
\right.
\end{equation}
where $\gamma^{\sigma}= \int_{\mathbb{R}^d \setminus B_1(0)} \frac{1}{\vert z \vert^{d +\sigma}} dz$.

Moreover, for $\sigma = 0$  we define $\mathcal{\mathcal{F}}^{0}: \mathcal{M}_{b}(\mathbb{R}^d) \rightarrow \overline{\mathbb{R}}$ as
\begin{equation}\label{desi0}
\hat{\mathcal{F}}^{0}(\mu):=
\left\{
\begin{aligned}
& P^{0}(E) & \text{if $\mu = \chi_{E} \mathcal{L}^d$,} \\
& + \infty  & \text{elsewhere.} 
\end{aligned}
\right.
\end{equation}
The following theorem has been proved in \cite{DeLucaNovPons}.

\begin{theorem}\label{col} 
	The following compacntess and $\Gamma$-convergence results hold. 
	\begin{itemize}
		\item[]{\bf Compactness:}
	Let $\sigma \in [0,1)$  and let $r_n\to 0^+$.
	Let $U\subset\R^d$ be an open bounded set and let $\{E_n\}_{n\in\N}\subset \M_f(\R^d)$ be such that $E_n\subset U$  for all $n\in\N$\,. Finally, let $C>0$\,.
	
	If $\jj_{r_n}^{\sigma}(E_n)\le C$  for all $n\in\N$,
	then, up to a subsequence, $\chi_{E_n} \to \chi_E$ in $L^1(\R^d)$ for some $E\in \M_f(\R^d)$.   
	\item[]{\bf  $\Gamma$-convergence:}
	The following $\Gamma$-convergence result holds true.
	\begin{itemize}
		\item[(i)] ($\Gamma$-liminf inequality) For every $E\in\M_f(\R^d)$ and for every sequence $\{E_n\}_{n\in\N}$ with $\chi_{E_n}\to\chi_{E}$ strongly in $L^1(\R^d)$ it holds
		\begin{equation*}
		\hat{\mathcal F}^{\sigma}(E)\le\liminf_{n\to +\infty}\jj_{ r_n}^{\sigma}(E_n)\,.
		\end{equation*}
		\item[(ii)] ($\Gamma$-limsup inequality) For every $E\in\M_f(\R^d)$\,, there exists a sequence $\{E_n\}_{n\in\N}$ such that $\chi_{E_n}\to\chi_{E}$ strongly in $L^1(\R^d)$ and
		\begin{equation*}
		\hat{\mathcal F}^{\sigma}(E)\ge \limsup_{n\to +\infty}\jj_{r_n}^{\sigma}(E_n)\,.
		\end{equation*}
	\end{itemize}  
\end{itemize}  \end{theorem}

Next proposition provides error estimates comparing the discrete functionals  $\mathcal{F}_{\varepsilon}^{\sigma}$ with its continuous counterpart $J_{r_\e}^\sigma$.

\begin{proposition}\label{stafond}
	Let $\sigma \in [0,1)$, and let $\{\mu_{\varepsilon}\}_{\varepsilon \in (0,1)}\subset \mathcal{EM}$ be  such that 
	$\mu_\e\in \mathcal{EM}_\e$ for all $\e\in(0,1)$ and $\frac{\varepsilon^d \omega_d}{C^d} \mu_{\varepsilon}(\mathbb{R}^d)\leq M$  for some $M>0$.
	
	Then, there exists
	$\{E_{\varepsilon}\}_{\varepsilon \in (0,1)} \subset \M_{f}(\mathbb{R}^d)$ such that the following properties hold:
	\begin{itemize}
		\item[(i)] $\frac{\e^d \omega_d}{C^d} \mu_{\varepsilon} - \chi_{E_\e} \weakstar 0$ as $\e\to 0$;
		\item[(ii)] 
		$\vert \vert E_{\varepsilon} \vert - \frac{\varepsilon^d \omega_d}{C^d} \mu_{\varepsilon}(\mathbb{R}^d) \vert 
		\le C(M,d) \frac{\sqrt \e}{ \sqrt r_\e}$;
		\item[(iii)] $\vert \mathcal{F}_{\varepsilon}^{\sigma}(\mu_{\varepsilon})-{J}_{r_{\varepsilon}}^{\sigma}(E_{\varepsilon}) \vert \le C(\sigma,d, M)|\gamma^\sigma_{r_\e}| \frac{\sqrt \e}{\sqrt{r_\e}}$.	 
	\end{itemize}
	In particular, as a consequence of \eqref{condr}, we have  
	\begin{itemize}
		\item[(iii')]
		$\vert \mathcal{\hat F}_{\varepsilon}^{\sigma}(\mu_{\varepsilon})-{\hat J}_{r_{\varepsilon}}^{\sigma}(E_{\varepsilon}) \vert \to 0  \quad \text{ as } \e \to 0$.  
	\end{itemize}
	Vice-versa, if $\{E_\e\}_{\e\in(0,1)}\subset \M_f(\R^d)$ is
	such that $\vert E_\e \vert \leq M$
	for some $M>0$, then there exists $\{\mu_\e\}_{\e\in(0,1)}\subset \mathcal{EM}$ with $\mu_\e\in \mathcal{EM}_\e$ for all $\e\in(0,1)$ and such that (i),  (ii), (iii) and (iii') hold.  
\end{proposition}

\begin{proof}
	For every $\e>0$, set $\rho_\e:= \sqrt{\e r_\e}$.   Let $Q:= [0,1)^{d}$ and set
	\begin{equation*}
	\mathfrak{Q}^{\rho_{\varepsilon}}:= \{\rho_{\varepsilon}(Q+v),\; v\in \mathbb{Z}^d  \}.
	\end{equation*}	
	Let moreover
	$$
	\mathfrak{P}^{\rho_{\varepsilon}}:=\{q\in \mathfrak{Q}^{\rho_{\varepsilon}}: 
	\frac{\e^d\omega_d}{C^d}\mu_{\varepsilon}(q) \ge \rho_\e^d \}.
	$$
	For all $q\in \mathfrak{Q}^{\rho_{\varepsilon}}$ we denote by $\tilde q$  the square concentric to $q$ and such that
	$\tilde q=q$ if $q\in \mathfrak{P}^{\rho_{\varepsilon}}$, 
	while $ |\tilde q| = \frac{\e^d\omega_d}{C^d} \mu_\e(q)$ if $q\in  
	\mathfrak{Q}^{\rho_{\varepsilon}}\setminus \mathfrak{P}^{\rho_{\varepsilon}}$.
	
	By Lemma \ref{sviluppoasintoticodensità} and by easy scaling arguments  we deduce that
	\begin{equation} \label{berna2}
	\# 	\mathfrak{P}^{\rho_{\varepsilon}} \leq M \rho_\e^{-d},
	\qquad 
	0\le  \frac{\varepsilon^d \omega_d }{C^d} \mu_{\varepsilon}(q) - |\tilde q| \le C(d) \e \rho_\e^{d-1}
	\quad \text{ for all } q\in \mathfrak{Q}^{\rho_{\varepsilon}}.
	\end{equation}
	We define $E_{\varepsilon}:= \cup_{q\in\mathfrak{Q}^{\rho_{\varepsilon}}} \tilde q$.
	By \eqref{berna2} we have that 
	\begin{equation}\label{berna3}
	\vert \vert E_{\varepsilon} \vert - \frac{\varepsilon^d \omega_d}{C^d} \mu_{\varepsilon}(\mathbb{R}^d) \vert 
	\le M C(d) \frac{\e}{ \rho_\e} 
	=M C(d) \frac{\sqrt \e}{ \sqrt r_\e}, 
	\end{equation}	
	which proves property (ii). 
	
	Let us pass to the proof of (i).
	Given $\f\in\mathcal C^1_c(\R^d)$, by \eqref{berna2} we have
	\begin{equation}
	\Big |\langle	\frac{\e^d \omega_{d}}{C^d} \mu_{\varepsilon} - \chi_{E_\e},\f\rangle \Big| 
	\le C(d,M) \|\nabla \f\|_{L^\infty} \rho_\e + 
	\|\f\|_{L^\infty} C(d,M) \frac{\e}{\rho_\e},
	\end{equation}	
	which tends to $0$ as $\e\to 0$.
	
	We pass to the proof of (iii). 	
	First notice that by construction $\vert E_\e \vert \leq M+1$ for $\e$ small enough. Then, by rearrangement (see for instance Lemma A.6 of \cite{DeLucaNovPons}) it is easy to see that  $ -J^\sigma_{r_\e}(E_\e)\le  C(\sigma,d, M) |\gamma^\sigma_{r_\e}|$. 
Therefore, in order to prove (iii)	
	it is enough to show that
	\begin{align}\label{sta1}
	& - {J}_{r_{\varepsilon}}^{\sigma}(E_{\varepsilon}) \le - \mathcal{F}_{\varepsilon}^{\sigma}(\mu_{\varepsilon})  \big( 1 + C(\sigma,d)  \frac{\sqrt \e}{\sqrt r_\e}\big)
	+ C(\sigma,d) |\gamma_{r_\e}^\sigma| \frac{\sqrt \e}{\sqrt r_\e},
	\\
	\label{sta2}
	& - \mathcal{F}_{\varepsilon}^{\sigma}(\mu_{\varepsilon})
	\le - {J}_{r_{\varepsilon}}^{\sigma}(E_{\varepsilon}) \big( 1 + C(\sigma,d)  \frac{\sqrt \e}{\sqrt r_\e} \big)  + C(\sigma,d) |\gamma_{r_\e}^\sigma| \frac{\sqrt \e}{\sqrt r_\e}.		 
	\end{align}	
	We will prove only \eqref{sta2}, the proof of \eqref{sta1} being fully analogous. 
	For all $p, \, q\in  	\mathfrak{Q}^{\rho_{\varepsilon}}$ with $p\neq q$, set 
	\begin{align*}
	&I(p,q):= \{(x,y) \in \spt(\mu) \times \spt(\mu) \cap p\times q\},\\
	&R_\e(p,q) := \dist(p,q), 
	\qquad 
	\tilde R_\e(p,q) :=  \max_{x\in p, y\in q}\dist(x,y), 
	\qquad  m_\e(q):= 
	\frac{\e^d \omega_d}{C^d} \mu_\e(q).
	\end{align*}

	By \eqref{berna2} we have that 
	\begin{equation}\label{berna3}
	1\le \frac{m_\e(q)}{|\tilde q|}\le 1+C(d) \frac{\e}{\rho_\e}
	\qquad \text{ for all } q\in 
	\mathfrak{Q}^{\rho_{\varepsilon}}\, .
	\end{equation}
	Moreover, since $\tilde R_\e(p,q) \le R_\e(p,q) + C(d) \rho_\e$, it follows that
	there exists $C(\sigma,d)>0$ such that, for $\e$ small enough,
	\begin{equation}\label{berna4}
	\Big(\frac{\tilde R_\e(p,q)}{ R_\e(p,q)}\Big)^{d+\sigma} \le \big(1+ C(\sigma,d) \frac{\rho_\e}{R_\e(p,q)}\big) 
	\quad \text{for all } q, p\in 
	\mathfrak{Q}^{\rho_{\varepsilon}}: \, R_\e(p,q)\neq 0.
	\end{equation}
	Moreover, let 
	\begin{align*}
	& \mathscr Q^+:= \{(p,q) \in \mathfrak{Q}^{\rho_{\varepsilon}} \times \mathfrak{Q}^{\rho_{\varepsilon}} : R_\e(p,q) > r_\e \};\\
	&\mathscr Q^-:= \{(p,q) \in \mathfrak{Q}^{\rho_{\varepsilon}} \times \mathfrak{Q}^{\rho_{\varepsilon}} : \tilde R_\e(p,q) < r_\e \};\\
	&\mathscr  Q^=:=  \mathfrak{Q}^{\rho_{\varepsilon}} \times \mathfrak{Q}^{\rho_{\varepsilon}} \setminus (\mathscr Q^+ \cup \mathscr Q^-).
	\end{align*}
	Recalling that $\frac{\varepsilon^d \omega_d}{C^d} \mu_{\varepsilon}(\mathbb{R}^d)\leq M$ and \eqref{berna2}, it easily follows that, for $\e$ small enough 
	\begin{multline}\label{scoerr}
	(\frac{\e^d \omega_d}{C^d})^2 \sum_{(p,q)\in \mathscr Q^=} \sum_{(x,y)\in I(p,q)} |x-y|^{-d-\sigma}
	\\
	\le C(\sigma,d) r_\e^{-d-\sigma} r_\e^{d-1} \rho_\e = 
	C(\sigma,d ) r_\e^{-\sigma}  \frac{\sqrt \e}{\sqrt r_\e}
	\le C(\sigma,d) |\gamma_{r_\e}^\sigma|  \frac{\sqrt \e}{\sqrt r_\e}
	.
	\end{multline}
	By \eqref{berna3}, \eqref{berna4} and \eqref{scoerr} we have that, for $\e$ small enough, 
	\begin{multline*}
	- \mathcal{F}_{\varepsilon}^{\sigma}(\mu_{\varepsilon}) \le
	(\frac{\e^d \omega_d}{C^d})^2 \sum_{(p,q)\in \mathscr Q^+}
	\sum_{(x,y)\in I(p,q)} |x-y|^{-d-\sigma}
	+ C(\sigma,d) |\gamma_{r_\e}^\sigma|  \frac{\sqrt \e}{\sqrt r_\e}
	\\
	\le \sum_{(p,q)\in \mathscr Q^+} 
	m_\e(p) m_\e(q) R_\e(p,q)^{-d-\sigma}
	+ C(\sigma,d) |\gamma_{r_\e}^\sigma|  \frac{\sqrt \e}{\sqrt r_\e}
	\\
	\le \sum_{(p,q)\in \mathscr Q^+} \Big(1+C(d) \frac{\e}{\rho_\e}\Big)^2 \Big(1+  C(\sigma,d) \frac{\rho_\e}{R_\e(p,q)}\Big)  
	|\tilde p| |\tilde q| \tilde R_\e(p,q)^{-d-\sigma} 
	+ C(\sigma,d) |\gamma_{r_\e}^\sigma|  \frac{\sqrt \e}{\sqrt r_\e} 
	\\
	\le\Big(1+C(\sigma,d) \frac{\sqrt \e}{\sqrt r_\e} \Big)
	(-{J}_{r_{\varepsilon}}^{\sigma}(E_{\varepsilon}) )
	+ C(\sigma,d,M) |\gamma_{r_\e}^\sigma|  \frac{\sqrt \e}{\sqrt r_\e}.
	\end{multline*}
	Finally, property (iii') is an easy consequence of properties (ii), (iii) and of \eqref{condr}.
	
	The proof of the final claim of the proposition is fully analogoug to the  proof of the first part of the proposition.
\end{proof}

\subsection{Compactness and $\Gamma$-convergence}
Here we prove $\Gamma$-convergence and compactness properties for the functionals 
$\hat{\mathcal{F}}_{\varepsilon}^{\sigma}$
defined in \eqref{desi} and \eqref{desi0}. 
Conversely to what done for the integrable case $\sigma\in (-d,0)$, here we will present only the basic case,  assuming as in \cite{DeLucaNovPons}
that there are no forcing terms; we enforce  compactness assuming that the empirical measures have uniformly bounded support.
 
\begin{theorem}\label{gacofinal}
	Let $\sigma \in [0,1)$. The following compactness and  $\Gamma$-convergence results hold.
	\begin{itemize}
		\item[]{\bf Compactness:}
		Let $U \subset \mathbb{R}^d$ be an open bounded set and let $M>0$.  Let $\{\mu_{\varepsilon}\}_{\varepsilon\in(0,1)} \subset \mathcal{M}_{b}(\mathbb{R}^d)$ be 
	such that 
	\begin{equation}
	\hat{\mathcal{F}}_{\varepsilon}^{\sigma}(\mu_{\varepsilon})\leq M, \; \; \qquad 
	\spt(\mu_{\varepsilon}) \subset U \; \; \qquad \forall \varepsilon \in (0,1) \,. \label{pinco1}
	\end{equation}
	Then, $ \frac{\varepsilon^d \omega_d}{C^d} \mu_{\varepsilon} \rightarrow \chi_{E} \mathcal L^d$ tightly, as $ \varepsilon \rightarrow 0^{+}$,
	for some measurable set  $E \subset U$.
	\vskip4pt
			\item[]{\bf $\Gamma$-convergence:}
			The following $\Gamma$-convergence result holds true.	
	\begin{enumerate}
		\item ($\Gamma$-liminf inequality) For every $E \in \mathcal{M}_{f}(\mathbb{R}^d)$ and for every $\{\mu_{\varepsilon}\}_{\varepsilon \in (0,1)} \subset  \mathcal{M}_{b}(\R^d)$  with $\frac{\varepsilon^d \omega}{C^d}\mu_{\varepsilon} \rightarrow \chi_{E} \mathcal{L}^d$ tightly in $\mathcal{M}_{b}(\mathbb{R}^d)$, we have 
		\begin{equation*}
		\hat{\mathcal{F}}^{\sigma}(\chi_{E}\mathcal{L}^d) \leq \liminf_{\varepsilon \rightarrow 0^+} \hat{\mathcal{F}}_{\varepsilon}^{\sigma}(\mu_{\varepsilon}).
		\end{equation*}
		\item ($\Gamma$-limsup inequality) For every $E \in \mathcal{M}_{f}(\mathbb{R}^d)$, there exists a sequence $\{\mu_{\varepsilon}\}_{\varepsilon \in (0,1)}$  with $\mu_{\varepsilon} \in \mathcal{EM}_{\varepsilon}$ for all $\varepsilon \in (0,1)$ such that $\frac{\varepsilon^d \omega_d}{C^d} \mu_{\varepsilon} \rightarrow \chi_{E}\mathcal{L}^d$ tightly in $\mathcal{M}_{b}(\mathbb{R}^d)$ and 
		\begin{equation*}
		\hat{\mathcal{F}}^{\sigma}(\chi_{E}\mathcal{L}^d) \geq \limsup_{\varepsilon \rightarrow 0^+} \hat{\mathcal{F}}_{\varepsilon}^{\sigma}(\mu_{\varepsilon}).
		\end{equation*}
	\end{enumerate}
\end{itemize}
\end{theorem}

\begin{proof}
In order to prove the compactness property, first notice that by \eqref{pinco1} we deduce  that $\mu_{\varepsilon} \in \mathcal{EM}_{\varepsilon}$ for all $\varepsilon \in (0,1)$. From Proposition \ref{stafond} we obtain that there exists $\varepsilon_0 \in (0,1)$ such that
\begin{equation*}
\vert \hat{\mathcal{F}}_{\varepsilon}^{\sigma}(\mu_{\varepsilon})-\hat{J}_{r_\varepsilon}^{\sigma}(E_\varepsilon) \vert < 1 \quad \forall \varepsilon<\varepsilon_{0},
\end{equation*}
where $\{E_{\varepsilon}\}_\e$ is exactly the sequence of sets provided by Proposition \ref{stafond}. 
We deduce that $\hat{J}_{r_\varepsilon}^{\sigma}(E_\varepsilon)$ is bounded;
by  Theorem \ref{col} there exists $E \in \mathcal{M}_{f}(\mathbb{R}^d)$ such that, up to a subsequence, $\chi_{E_{\varepsilon}} \rightarrow \chi_{E}$ in $\mathrm{L}^1$ for $\varepsilon \rightarrow 0^+$.  Therefore, again by Proposition \ref{stafond}  $\frac{\varepsilon^d \omega_d}{C^d}\mu_{\varepsilon} \rightarrow \chi_{E}$ tightly as $ \varepsilon \rightarrow 0^+$.

Let us pass to the proof of the $\Gamma$-liminf inequality. By  Proposition \ref{stafond} 
and by Theorem \ref{col}  we obtain that
\begin{equation*}
\begin{split}
\hat{\mathcal{F}}^{\sigma}(\chi_{E}\mathcal{L}^d) \leq & \liminf_{\varepsilon \rightarrow 0^+} \hat{J}_{r_{\varepsilon}}^{\sigma}(\mu_{\varepsilon})  \\
\leq & \liminf_{\varepsilon \rightarrow 0^+} (\hat{J}_{r_{\varepsilon}}^{\sigma}(E_{\varepsilon})- \hat{\mathcal{F}}_{\varepsilon}^{\sigma}(\mu_{\varepsilon}))+ \liminf_{\varepsilon \rightarrow 0^+} \hat{\mathcal{F}}_{\varepsilon}^{\sigma}(\mu_{\varepsilon})  \\
\leq & \liminf_{\varepsilon \rightarrow 0^+} \hat{\mathcal{F}}_{\varepsilon}^{\sigma}(\mu_{\varepsilon}).
\end{split}
\end{equation*}
Hence the $\Gamma$-liminf inequality holds. 

We now prove the $\Gamma$-limsup inequality. 
Let $\{E_\e\}_\e$ be the recovery sequence provided by Theorem \ref{col};  we have 
$$
\jj^\sigma_{r_\e}(E_\e) \to  \hat{\mathcal F}^\sigma(\chi_E \mathcal{L}^d) \qquad \text{ as } \e\to 0.
$$
Let now $\{\mu_\e\}_{\e\in(0,1)}$ be the sequence provided by the second part of Proposition \ref{stafond}. Then, we have  
\begin{equation*}
 \vert \hat{\mathcal{F}}_{\varepsilon}^{\sigma}(\mu_{\varepsilon})- \hat{\mathcal F}^\sigma(\chi_E \mathcal{L}^d)\vert \leq 
 \vert \hat{\mathcal{F}}_{\varepsilon}^{\sigma}(\mu_{\varepsilon})-\hat{J}_{r_{\varepsilon}}^{\sigma}(E_\e) \vert 
 +   \vert \hat{J}_{r_{\varepsilon}}^{\sigma}(E)- \hat{\mathcal F}^\sigma(\chi_E \mathcal{L}^d)\vert,
\end{equation*}
which, in view of Proposition \ref{stafond}(iii'), tends to $0$ as $\e\to 0$.
\end{proof}
\section*{Acknowledgements}
The authors are members of the Gruppo Nazionale per l'Analisi Matematica, la Probabilit\`a e le loro Applicazioni (GNAMPA) of the Istituto Nazionale di Alta Matematica (INdAM).

Tha authors thank L. De Luca and M. Novaga for useful discussions at the early stage of this project.


\end{document}